\documentclass[11pt, a4paper]{article}
\usepackage{amsmath}
\usepackage{amssymb, units, amsthm}
\usepackage[english]{babel}

  \usepackage{paralist}
  \usepackage{graphics} 
  \usepackage{epsfig} 
\usepackage{graphicx}  \usepackage{epstopdf}
 \usepackage[colorlinks=true]{hyperref}
\hypersetup{urlcolor=blue, citecolor=red}

\def \R{\mathbb{R}}

\newtheorem{theorem}{Theorem}[section]

\newtheorem{lemma}[theorem]{Lemma}

\newtheorem{definition}[theorem]{Definition}

\title
      {Mathematical analysis of weak and strong solutions to an evolutionary model for magnetoviscoelasticity}

\def\eR{\mathbf{R}}

\def\eN{\mathbf{N}}

\newcommand{\ext}{\mathrm{ext}}
\newcommand{\eff}{\mathrm{eff}}
\newcommand{\res}{\mathrm{res}}

\def\d{\; \mathrm{d}}
\def\dx{\; \mathrm{d}x}

\def\dt{\; \mathrm{d}t}
\def\ds{\; \mathrm{d}s}

\def\dvr{\mathop{\mathrm{div}}\nolimits}

\def\tder{\partial_t}

\def\R{\mathbb{R}}

\begin{document}
\author{Martin Kalousek, Joshua Kortum and Anja
Schl\"omerkemper\footnote{Corresponding author:  anja.schloemerkemper@mathematik.uni-wuerzburg.de}}

\maketitle
{\footnotesize
 \centerline{University of W\"urzburg, Institute of Mathematics, Emil-Fischer-Stra\ss e 40,}
   \centerline{97074 W\"urzburg, Germany}
} 

\medskip


\bigskip


\begin{abstract}
The paper is concerned with the analysis of an evolutionary model for magnetoviscoelastic materials in two dimensions. The model consists of a Navier-Stokes system featuring a dependence of the stress tensor on elastic and magnetic terms, a regularized system for the evolution of the deformation gradient and the Landau-Lifshitz-Gilbert system for the dynamics of the magnetization. \\ 
First, we show that our model possesses global in time weak solutions, thus extending work by Bene{\v s}ov\'a et al.\ 2018. Compared to that work, we include the stray field energy and relax the assumptions on the elastic energy density. Second, we prove the local in time existence of strong solutions. Both existence results are based on the Galerkin method. Finally, we show a weak-strong uniqueness property.
\end{abstract}

\section{Introduction}

The system of partial differential equations considered in this article models the evolution of magnetoviscoelastic materials. Through the coupling of elastic and magnetic effects and several nonlinearities, the proof of existence and uniqueness of solutions is a challenge. We will provide global-in-time existence of weak solutios, local-in-time existence of strong solutions and weak-strong uniqueness. 

The coupling of magnetic and elastic interactions is observed in several solid and fluid materials such as giant magnetostrictive materials or magnetorheological fluids, which happen to have various technological applications as,  e.g., sealings or shock absorbers. A fundamental difficulty is that one needs to find a way to model elasticity (which is usually phrased in Lagrangian coordinates) and magnetism (which is modeled in Eulerian coordinates) in a mathematically rigorous way. Following \cite{BeFoLiSc,Forster,linliuzhang,LiuWalkington2001}, we apply a variational approach that allows us to unify the different coordinates. 

In the literature, evolutionary models for magnetoelastic materials were considered under the assumption of quasistatic evolution \cite{KruzikStefanelliZeman2015} or in the setting of small strain, see,  e.g., \cite{CarbouEfendievFabrie2011,CCVVC,ChipotShafrir_etal2009,EET}.

We consider the following system of partial differential equations for the velocity $v: (0,T)\times \Omega  \to \R^2$, the deformation gradient $F: (0,T)\times \Omega \to \R^{2\times 2}$ and the magnetization $M: (0,T)\times \Omega  \to \R^3$, where $\Omega \subset \R^2$ is a bounded $C^\infty$-domain and $T>0$:
\begin{align}
	\tder v+(v\cdot\nabla)v+\nabla p-\Delta v&=\dvr(W'(F)F^T)-\dvr(\nabla M\odot\nabla M) +\nabla^T H_\ext M, \label{NSSystem}\\
	\dvr v&=0,\\
	\tder F+(v\cdot\nabla)F-\nabla vF&=\kappa\Delta F,\label{DefGradSystem}\\
	\tder M+(v\cdot\nabla)M&=\Delta M+|\nabla M|^2M   -M\times H_\eff- M\times (M \times 
	H_\res), \label{MagnetSystem} 
\end{align}
where $W$ is the elastic energy density and $H_\eff = \Delta M - (M\cdot e_3) e_3 + H_\ext$ with $e_3$ being the third canonical basis vector of $\R^3$, and $H_\res = H_\eff - \Delta M$ in $(0,T)\times\Omega$. This model was derived in a variational approach starting from the kinetic energy, the elastic energy and the micromagnetic energy; we refer to  \cite{BeFoLiSc} for an overview and to \cite{Forster} for further details. 

In this article we extend the previous work of \cite{BeFoLiSc} and \cite{Forster} on the mathematical analysis
of the considered equations. In Theorem~\ref{Thm:WeakSol}, we prove that global-in-time weak solutions exist under a more general assumption on the elastic energy density~$W$: We require $W$ to be convex instead of strong convexity. This is a particular  consequence of Ladyzhenskaya's inequality (see \eqref{addestimate}) which allows to control $|F|^2$ independently of the elastic energy. We are also able to include the effect of the stray field energy represented by the term $-(M\cdot e_3)e_3$
in the effective magnetic field $H_\eff$. 

We prove weak-strong uniqueness in Theorem~\ref{Thm: WSUniqueness}. It is thus of interest to construct local-in-time strong solutions for smooth enough initial data, see Theorem~\ref{Thm:Main}. These serve as proper candidates for a comparison with the weak solutions leading to a weak-strong uniqueness result.

Uniqueness in two dimensions was proven for an approximative system  in \cite{SchlZab}. The Landau-Lifshitz-Gilbert equation was substituted by the gradient flow equation
$$ \partial_t M + (v\cdot \nabla) M = \Delta M + \frac{1}{\mu^2} (1-|M|^2) M$$
with $\mu>0$.
Existence of weak solutions for this gradient flow system is established in \cite{Forster}. Recently, Zhao \cite{Zhao2018} proved the existence of strong close-to equilibrium solutions to the gradient flow system in the periodic setting with $\kappa=0$ and the simplifying assumption $W(F)= \frac12 |F|^2$.

From a more general perspective, the model \eqref{NSSystem}--\eqref{MagnetSystem} encompasses multiple difficulties of the related models for complex fluids. Only taking into account the momentum equation and the Landau-Lifshitz-Gilbert equation, the system is related to the  Ericksen-Leslie model for liquid crystal flows. For this, uniqueness of weak solutions was shown by Lin and Wang in \cite{lin}. The transport equation for $F$  is similar as the ones for viscoelastic flows. There are very few results dealing with existence theory. Existence of strong solutions to \eqref{NSSystem}--\eqref{DefGradSystem} ($M=0$) was proved in 
\cite{linliuzhang} for small and smooth enough initial data. To the authors' knowledge, there does not exist a sufficient existence theory for weak solutions with $\kappa =0$ so far. An existence proof for dissipative solutions to the system with $\kappa=0$ is given in \cite{Kalousek}. From this point of view, the mathematical assumption $\kappa >0$ (but small) is necessary so far to ensure the existence of weak solutions.

The outline of this article is as follows: In Section 2 we introduce the notation and state the main theorems, i.e., global-in-time existence of weak solutions (Theorem \ref{Thm:WeakSol}), local-in-time existence of strong solutions (Theorem \ref{Thm:Main}) and the 
weak-strong uniqueness (Theorem \ref{Thm: WSUniqueness}). Section 3 is concerned with the proof of the existence theorems while the weak-strong uniqueness property is proven in Section 4. The Appendix contains some often used inequalities.

\section{Statement of the main results}

We analyze the system \eqref{NSSystem}--\eqref{MagnetSystem} accompanied with the boundary and initial conditions
\begin{equation}\label{BoundInitCond}
\begin{alignedat}{2}
	v=&0 &&\text{ on }(0,T)\times\partial\Omega,\\
	F=&0 &&\text{ on }(0,T)\times\partial\Omega,\\
	(\nabla M)n=&0 &&\text{ on }(0,T)\times\partial\Omega,\\
	v(0,\cdot)=&v_0 &&\text{ in }\Omega,\\
	F(0,\cdot)=&F_0 &&\text{ in }\Omega,\\
	M(0,\cdot)=&M_0 &&\text{ in }\Omega,
\end{alignedat}
\end{equation}
where $n$ denotes the outer normal to $\Omega$, and the functions $v_0$ and $ M_0$ satisfy $\dvr v_0=0$, $|M_0|\equiv 1$  in $\Omega$. 

We begin this section with the introduction of the notation. Generic constants are denoted by $c$. Let $A,B\in\eR^{d\times d}$; then $A\odot B$ denotes the matrix $A^\top B$. Let $a\in \eR^m, b\in\eR^n$; then $a\otimes b$ denotes an $m\times n$ matrix having the product $a_ib_j$ on the entry $i,j$.  For $k\in\eN$ and $q\in[1,\infty]$, $L^q(\Omega)$ and $W^{k,q}(\Omega)$ denote the standard Lebesgue and Sobolev spaces. In order to keep the notation short, we write e.g. $ \|\cdot \|_q$ instead of $\|\cdot \|_{L^q(\Omega)^m}$, $ \|\cdot \|_{k,q}$ instead of $\|\cdot \|_{W^{k,q}(\Omega)^{m\times n}}$, and $\|\cdot \|_{L^q(0,t;L^r(\Omega))}$ instead of $\|\cdot \|_{L^q(0,t;L^r(\Omega)^{m})}$, etc. We set
\begin{align*}
	L^2_{\dvr}(\Omega)^2&=\overline{\{u\in C^\infty_c(\Omega)^2:\ \dvr u=0 \text{ in }\Omega\}}^{\|\cdot\|_2},\\
	W^{1,2}_{0,\dvr}(\Omega)^2&=\overline{\{u\in C^\infty_c(\Omega)^2:\ \dvr u=0 \text{ in }\Omega\}}^{\|\cdot\|_{1,2}},\\
	W^{2,2}_n(\Omega)^3&=\{u\in W^{2,2}(\Omega)^3:(\nabla u) \operatorname{n}=0\text{ on }\partial\Omega\}.
\end{align*}
As already mentioned, we show the existence of weak and  strong solutions to the system \eqref{NSSystem}--\eqref{MagnetSystem} in two dimensions. To this end, we first define what we mean by these notions.

\begin{definition}\label{NotionSol}
Let $\Omega \subset \eR^2$ be a $C^\infty$-domain and let $T>0$. Then we call $(v,F,M)$ enjoying the regularity 
\begin{align*}
	v & \in L^\infty(0,T; L^2_{\dvr}(\Omega)^2) \cap L^2(0,T; W_{0,\dvr}^{1,2}(\Omega)^2),\\
	F & \in L^\infty(0,T; L^2(\Omega)^{2 \times 2}) \cap L^2(0,T; W_{0}^{1,2}(\Omega)^{2\times 2}),\\
	M & \in L^\infty(0,T; W^{1,2}(\Omega)^3) \cap L^2(0,T; W^{2,2}(\Omega)^3)
\end{align*}
a \emph{weak solution} of the system \eqref{NSSystem}--\eqref{MagnetSystem} accompanied with initial/boundary conditions \eqref{BoundInitCond} if it satisfies the boundary value condition in the sense of traces as well as the initial conditions
\begin{equation}\label{ini-cond}
v (t,\cdot)\xrightarrow{L^2(\Omega)} v_0(\cdot), \qquad F (t,\cdot)\xrightarrow{L^2(\Omega)} F_0(\cdot), \qquad M (t,\cdot)\xrightarrow{W^{1,2}(\Omega)} M_0(\cdot),
\end{equation}
and if it fulfills the system 
\begin{equation}\label{WForm}
	\begin{split}
	&\int_0^T \int_\Omega -v \cdot \tder \phi + (v\cdot\nabla)v\cdot\phi-\left(\nabla M\odot\nabla M-W'(F)F^\top-\nabla v\right)\cdot\nabla\phi \\
	& \qquad \qquad -(\nabla H_\ext^\top M)\cdot\phi\dx\dt= \int_\Omega v_0 \cdot \phi(0,x)\dx,\\
	&\int_0^T \int_\Omega -F \cdot \tder \xi +(v\cdot\nabla)F\cdot\xi-(\nabla vF)\cdot\xi+\kappa \nabla F\cdot\nabla\xi\dx\dt \\
	& \qquad \qquad  = \int_\Omega F_0 \cdot \xi(0,x) \dx, \\
	&\tder M + (v\cdot \nabla) M = \Delta M + |\nabla M|^2 M - M \times H_\eff - 
	M\times  (M \times H_\res)
	\end{split}
\end{equation}
for all $\phi(t,x) = \phi_1(t)\phi_2(x)$ with $\phi_1 \in W^{1,\infty}(0,T)$ satisfying $\phi_1(T)=0$ and $\phi_2 \in W^{1,2}_{0,\dvr}(\Omega)^2$, for all $\xi(t,x) = \xi_1(t)\xi_2(x)$ with $\xi_1 \in W^{1,\infty}(0,T)$ satisfying $\xi_1(T)=0$ and $\xi_2 \in W_0^{1,2}(\Omega)^{2 \times 2}$, and the last equation a.e.\ on $\Omega \times (0,T)$.
\end{definition}
\begin{definition}
A triple $(v,F,M)$ is called a \emph{strong solution} to  the system \eqref{NSSystem}--\eqref{MagnetSystem} on $(0,T)$ with an associated pressure $p$ if it is a weak solution and enjoys the regularity
\begin{eqnarray*}
	 v & \in & L^\infty (0,T;W^{1,2}_{0\dvr}(\Omega)^{2})\cap L^2(0,T;W^{2,2}(\Omega)^{2}),\\ 
	\tder v & \in & L^2(0,T;L^2(\Omega)^2), \\
	 F &  \in &L^\infty (0,T;W_0^{1,2}(\Omega)^{2\times 2})\cap L^2(0,T;W^{2,2}(\Omega)^{2\times 2}),\\
	\tder F &\in & L^2(0,T;L^2(\Omega)^{2\times 2}),\\
	 M & \in &L^\infty(0,T;W^{2,2}_n(\Omega)^3)\cap L^2(0,T;W^{3,2}(\Omega)^3),\\
	 \tder M&\in& L^2(0,T;W^{1,2}(\Omega)^3),\\
	p & \in &L^2(0,T;W^{1,2}(\Omega)), \int_\Omega p(t,x)\dx=0,
\end{eqnarray*}
and equations \eqref{NSSystem}, \eqref{DefGradSystem} and \eqref{MagnetSystem} are satisfied a.e.\ in $(0,T)\times\Omega$. 
\end{definition}

In order to state precisely our main theorems we need to specify certain aspects of the data. At first, the elastic energy density $W: \eR^{2 \times 2} \to \eR$ should be a convex $C^2$-function satisfying the following growth conditions:
\begin{itemize}
	\item There exists a positive constant $C_1$ such that
				\begin{equation} C_1 (|A|^2-1) \leq W(A) \leq C_1 (|A|^2 +1), \qquad \forall A \in \eR^{2 \times 2}.
				\label{W1} \end{equation}
	\item There exists a positive constant $C_2$ such  that
				\begin{equation} |W'(A_1)-W'(A_2)| \leq C_2|A_1 - A_2|, \qquad \forall A_1, A_2 \in \eR^{2 \times 2},
				\label{W2} \end{equation}
				i.e., such that the derivative is Lipschitz continuous.
	\item There exists a positive constant $C_3$ such that 
				\begin{equation} |W''(A)|\leq C_3, \qquad \forall A \in \eR^{2 \times 2}.
				\label{W3} \end{equation}
\end{itemize}
Note that the growth conditions are not independent of each other. The convexity and differentiability of $W$ together with the first growth condition imply the other two. For technical reasons we also require $W'(0)=0$. This is also a consequence if we require the frame-indifference of $W$, i.e., $W(QF)=W(F)$ for all $Q \in SO(2)$ and all $F \in \eR^{2 \times 2}$.

With this in mind, our first theorem states the existence of global-in-time weak solutions.
\begin{theorem}[Existence of weak solution] \label{Thm:WeakSol}
Let $\Omega \subset \eR^2$ be a $C^\infty$-domain and let $T > 0$ be the final time of the evolution. Let $W \in C^2(\eR^{2\times 2})$ be convex and  satisfy  \eqref{W1}--\eqref{W3}. In addition, assume that
\begin{align*}
	H_\ext &\in C([0,T];L^2(\Omega)^3)\cap L^2(0,T;L^\infty(\Omega)^3)\cap L^3(0,T;W^{1,4}(\Omega)^3),\\
	\tder H_\ext&\in L^1(0,T;L^1(\Omega)^3)
\end{align*}
and $v_0\in L^2_{\mathrm{div}}( \Omega)^2$, $F_0\in L^2(\Omega)^{2 \times 2}$ and $M_0\in W^{2,2}(\Omega)^3$ with $|M_0| \equiv 1$ a.e.\ on $\Omega$.
Moreover, let the initial data and the external field satisfy the smallness condition 
\begin{align} \begin{split}\label{smallinitialdatacondThmB}
\int_\Omega &\frac{1}{2} |v_0|^2 + \frac{1}{2}|\nabla M_0|^2 + \frac12 (M_0 \cdot e_3)^2+ W(F_0) \dx \\
 & + 2\|H_\mathrm{ext}\|_{L^\infty(0,T; L^1(\Omega)^3)} + \|\partial_t H_\mathrm{ext}\|_{L^1(0,T; L^1(\Omega)^3)} < \frac{1}{\widetilde C} \end{split}
\end{align}
for a suitably large constant $\widetilde C>0$ depending just on $\Omega$. Then there exists a weak solution of the system \eqref{NSSystem}--\eqref{MagnetSystem} accompanied with initial/boundary conditions \eqref{BoundInitCond} in the sense of Definition \ref{NotionSol}.
\end{theorem}
Additionally, we have the existence of local-in-time strong solutions.
\begin{theorem}[Existence of a strong solution]\label{Thm:Main}
Let $\Omega\subset\eR^2$ be a $C^\infty$-domain and $T>0$. Let $W\in C^2(\eR^{2\times 2})$ be convex and satisfy (\ref{W1})--(\ref{W3}). In addition, assume that
\begin{align*}
	H_\ext &\in C([0,T];L^2(\Omega)^3)\cap L^2(0,T;L^\infty(\Omega)^3)\cap L^3(0,T;W^{1,4}(\Omega)^3),\\
	\tder H_\ext&\in L^1(0,T;L^1(\Omega)^3)
\end{align*}
and $v_0\in W^{1,2}_{0,\dvr}(\Omega)^2$, $F_0\in W^{1,2}_0(\Omega)^{2\times 2}$ and $M_0\in W^{2,2}_n(\Omega)^3$. Then there exists $T^\ast\in(0,T]$ and a strong solution to \eqref{NSSystem}--\eqref{MagnetSystem} on $\Omega \times (0,T^\ast)$ accompanied with boundary and initial conditions \eqref{BoundInitCond} in the sense of traces and the sense of \eqref{ini-cond}.
\end{theorem}

Our third result yields a comparison of these two different notions of solutions.
\begin{theorem}[Weak-strong uniqueness] \label{Thm: WSUniqueness}
Strong solutions of \eqref{NSSystem}--\eqref{MagnetSystem} subject to \eqref{BoundInitCond} are unique in the class of weak solutions emanating from the same initial data.
\end{theorem}

The three theorems will be proven in the following two sections. The proofs of Theorems \ref{Thm:WeakSol} and \ref{Thm:Main}, given in the subsequent section, are based on the proof of the  existence of a weak solution from \cite{BeFoLiSc}. While for the existence of weak solutions only the additional estimate 
\eqref{addestimate} is needed, we show that the sequence of Galerkin approximations converging to a weak solution is bounded uniformly in spaces that are required for the regularity of a strong solutions. Hence, the constructed weak solution is in fact a strong solution but not necessarily on the whole time interval on which the weak solution exists. We notice that the dependence of the elastic energy density on the deformation gradient assumed here extends the assumptions of \cite[Theorem 2]{BeFoLiSc}. \\
In section \ref{sec:uni}, we prove the weak-strong uniqueness result. The issue of proving weak-strong uniqueness, compared to uniqueness of weak solutions, is a typical approach when one may not expect global-in-time regular solutions. For the three-dimensional Navier-Stokes equations, uniqueness of weak solutions is a well-known open problem. However, it is known for the two-dimensional Navier-Stokes equations (see e.g.\ \cite[Chapter 3]{Robinsonetal}). Our system also encompasses the Landau-Lifshitz-Gilbert equation, for which uniqueness of weak solutions  $M \in L^2(0,T; W^{2,2}(\Omega)^3)$ is even a delicate problem in two dimensions (see \cite{lin} for the uniqueness of the related simplified Ericksen-Leslie model). Part of our future work will be concerned with proving uniqueness of weak solutions to our above system by using methods from harmonic analysis.

\section{Proof of Theorems \ref{Thm:WeakSol} and \ref{Thm:Main}}
Both existence proofs are based on the Galerkin method. Therefore, we start presenting them together. Large parts of this construction can be taken over from \cite[Theorem 3.2]{BeFoLiSc}; we thus refer to this and highlight the different arguments needed for our extension.

Let us begin by introducing the basis of the function spaces involved. By $\{\xi_i\}_{i=1}^\infty$ we denote an orthogonal basis of $W^{1,2}_{0,\dvr}(\Omega)^2$ that is orthonormal in $L^2_{\dvr}(\Omega)^2$ consisting of eigenfunctions of the Stokes operator with the homogeneous boundary condition. For $n\in\eN$, we set
\begin{equation*}  
	H^n=\text{span}\{\xi_1,\xi_2,\ldots,\xi_n\}.
\end{equation*}
Following the notation from \cite{BeFoLiSc}, we introduce, for $t_0\in(0,T]$ and $L=\|v_0\|_2+1$,
\begin{align*}
V_n(t_0)=\left\{v(t,x)=\sum_{i=1}^n g_n^i(t)\xi_i(x):[0,t_0)\times\Omega\to\eR^2; \sup_{t\in[0,t_0)}\sum_{i=1}^n|g_n^i(t)|^2\leq L^2, \right. \\ \left. g^i_n(0)=\int_\Omega v_0(x)\xi_i(x)\dx\right\},
\end{align*}
where the $g_n^i$ are Lipschitz solutions of the corresponding ordinary differential equations, see \cite[Definition~4.1]{BeFoLiSc} for details.

 The existence of a Galerkin approximation $v_n\in V_n(t^*)$ and a corresponding pair $(F_n,M_n)$ was shown in Step 1 of the proof of \cite[Theorem 2]{BeFoLiSc} combining the Schauder fixed point theorem with existence and regularity results for the related ordinary differential equations. Under our more general assumptions on $W$ and the addition of the stray field energy, the proof can be performed in the same way. We thus obtain the existence of a Galerkin approximation
$v_n\in V_n(t^*)$ for some $t^*\in (0,T)$ and fixed $n\in\eN$ satisfying
\begin{align}\label{NSFinDim}
	\int_\Omega \partial_t v_n\cdot \xi+(v_n\cdot\nabla)v_n\cdot\xi-\left(\nabla M_n\odot\nabla M_n-W'(F_n)(F_n)^\top-\nabla v_n\right)\cdot\nabla\xi  \notag \\ 
	-(\nabla H_\ext^\top M_n)\cdot\xi=0  \qquad  \text{ in }(0,t^*)\text{ for all }\xi\in H^n
	\end{align}
and a corresponding pair $(F_n,M_n)$ enjoying the regularity
\begin{equation*}
	\begin{split}
	&F_n\in L^\infty(0,t^*;L^2(\Omega)^{2\times 2})\cap L^2(0,t^*;W^{1,2}_0(\Omega)^{2\times 2}),\ \tder F_n\in L^2(0,t^*;(W^{1,2}_0(\Omega)^{2\times 2})^*),\\
	&M_n\in W^{1,\infty}(0,t^*;L^2(\Omega)^3)\cap W^{1,2}(0,t^*;W^{1,2}(\Omega)^3)\cap L^2(0,t^*;W^{3,2}(\Omega)^3)
	\end{split}
\end{equation*}
and satisfying
\begin{align}
	\left\langle\partial_t F_n,\Xi\right\rangle &+\int_\Omega(v_n\cdot\nabla)F_n\cdot\Xi-\int_\Omega\nabla v_nF_n\cdot\Xi+\kappa\int_\Omega\nabla F_n\cdot\nabla\Xi =0\notag\\&\text{ for all }\Xi\in W^{1,2}_0(\Omega)^{2\times 2}\text{ a.e.\ in }(0,t^*) ,\label{DGFinDimW}\\
	\partial_t M_n &+(v_n\cdot\nabla)M_n -\Delta M_n + M_n\times H_\eff  -|\nabla M_n|^2M_n \notag \\
	&+(M_n\cdot H_\res )M_n-H_\res=0\text{ a.e.\ in }(0,t^*)\times\Omega\label{LLGFinDim}
\end{align}
with $M_n$ fulfilling the constraint
\begin{equation}\label{MNDiscrBound}
	|M_n|=1\text{ a.e.\ in }(0,t^*)\times\Omega.
\end{equation} 

The existence proof for the strong solutions requires manipulations with $\Delta F^n$ in order to collect uniform estimates of higher order. A key ingredient for this is Lemma~\ref{KeyLemma} below, which yields, for a fixed $v_n\in V_n(t_0)$, higher regularity of the solution pair $(F,M)$ satisfying \eqref{DefGradSystem} and \eqref{MagnetSystem}.
Thus, assuming $F_0\in W^{1,2}_0(\Omega)^{2\times 2}$, we have the improved regularity of $F^n$ at hand, i.e., 
\begin{equation*}
F_n\in L^\infty(0,t^*;W^{1,2}_0(\Omega)^{2\times 2})\cap L^2(0,t^*;W^{2,2}(\Omega)^{2\times 2}),\ \tder F_n\in L^2(0,t^*;L^2(\Omega)^{2\times 2}),    
\end{equation*}
and 
\begin{equation}\label{DGFinDim}
    \partial_t F_n+(v_n\cdot\nabla)F_n-\nabla v_nF_n-\kappa\Delta F_n =0\text{ a.e.\ in }(0,t^*)\times\Omega
\end{equation}
instead of \eqref{DGFinDimW}.

Let us return to the Galerkin system for fixed $n$. Denote the energy of the system by 
$$ E(t) = \frac{1}{2} \left(\|v_n(t)\|_2^2  +2\|W(F_n(t))\|_1+\|\nabla M_n(t)\|_2^2 + \|(M_n)_3\|_2^2 \right).$$ 
Following the arguments of Step 2 of the proof of \cite[Theorem~3.2., p.~1215]{BeFoLiSc}, we derive, for $t\in(0,t^*)$, the energy inequality
\begin{equation}
    \begin{split} \label{ZerothApEst}
        E(t) &+ \int_0^t \|\nabla v_n\|_2^2+\kappa W''(F_n)\nabla F_n \cdot \nabla F_n \ds \\
        &\leq E(0) + 2 \|H_\ext\|_{L^\infty(0,T;L^1(\Omega))}+ \|\partial_t H_\ext \|_{L^1(0,T; L^1)},
    \end{split}
\end{equation}
which is needed for the a-priori bounds,
as well as
\begin{equation}\label{FirstApEst}
	\begin{split}
		E(t) &+ \int_0^t\|\nabla v_n\|_2^2+\kappa W''(F_n)\nabla F_n \cdot \nabla F_n +\|\Delta M_n\|_2^2\ds \\ 
		&\leq E(0) +\int_0^t\|\nabla M_n\|_4^4+c\|H_\ext \|_2^2 + c\|(M_n)_3\|_2^2 \ds.
	\end{split}	
\end{equation}
Since $W$ is convex, we deduce $W''(F)\nabla F \cdot \nabla F \geq 0$ a.e.\ in $(0,t^*)\times\Omega$ and the third term on the left hand side of (\ref{FirstApEst}) can be omitted. To receive a uniform bound on $\nabla F_n$, we test equation \eqref{DGFinDimW} by $F_n$ and obtain
\begin{equation}
    \begin{split} \label{addestimate}
\frac{1}{2}\frac{\d}{\dt} \| F_n\|_2^2 & + \kappa \|\nabla F_n\|_2^2 \leq \|F_n\|_4^2 \|\nabla v_n\|_2 \leq C \|F_n\|_2 \|\nabla F_n\|_2 \|\nabla v_n\|_2 \\
&\leq \frac{\kappa}{2} \|\nabla F_n\|_2^2 +  C \underbrace{\|F_n\|_2^2 \|\nabla v_n\|_2^2}_{\in L^1(0,t^*)}. 
    \end{split}
    \end{equation}
In combination with \eqref{ZerothApEst}, this implies a uniform $L^2$-bound on $\nabla F_n$. 
In order to estimate the term $\|\nabla M_n\|_4^4$, we employ inequality \eqref{LadMatrix}
to obtain
\begin{equation*}
	\|\nabla M_n\|_4^4\leq c\|\nabla M_n\|_2^2\|\Delta M_n\|_2^2
\end{equation*}
and arrive at 
\begin{equation} \label{InterMedIneq}
	\begin{split}
	E(t) &+\int_0^t\|\nabla v_n\|_2^2+\|\nabla F_n\|_2^2+\|\Delta M_n\|_2^2\ds\\&\leq E(0) +c\int_0^t\|\nabla M_n\|_2^2\|\Delta M_n\|_2^2+ \|H_\ext \|_2^2 + \|(M_n)_3\|_2^2 + 1\ds.
	\end{split}
\end{equation}
The latter inequality is crucial for collecting uniform estimates necessary for the conclusion of the proof of Theorem~\ref{Thm:WeakSol}. The arguments essentially follow those of the proof of \cite[Theorem 3.2]{BeFoLiSc}. Instead of repeating the details, we outline the basic steps: The smallness assumption on the initial data allows us to absorb the term on the right hand side of \eqref{InterMedIneq} involving $\|\Delta M_n\|_2^2$, which leads to the desired bound on $\Delta M_n$ in $L^2(0,T;L^2(\Omega)^3)$. 

Afterwards, we derive a-priori bounds on $\{(\partial_t v_n , \partial_t F_n, \partial_t M_n)\}_{n=1}^\infty$ via duality arguments. Employing these, we apply the Aubin-Lions lemma to obtain convergences that are strong enough to pass to the limit in \eqref{NSFinDim}--\eqref{LLGFinDim}. 

The attainment of the initial data in the sense of Definition~\ref{NotionSol} follows from the energy inequality. The only difference with respect to \cite[Step 6]{BeFoLiSc} is the strong attainment of initial data $F(t) \to F_0$ in $L^2(\Omega)^{2 \times 2}$ as $t \to 0^+$. This fact is a consequence of the weak formulation and Lemma \ref{Lem:TDer} below, i.e., we test the transport equation by $F$ yielding
$$ \frac{1}{2} \frac{\d}{\d t} \|F(t)\|_2^2 = - \kappa \| \nabla F(t) \|_2^2 + \int_\Omega ( \nabla v F \cdot F)(t). $$
As the right-hand side lies in $L^1(0,T)$, the function $t \mapsto\|F(t)\|_2^2 $ is absolutely continuous on $[0,T]$. Since $F(t) \rightharpoonup F_0$ in $L^2(\Omega)^{2 \times 2}$ for $t \to 0^+$, the claim follows. 
The boundary conditions are satisfied in the sense of traces due to the continuity of the trace operator. Hence the proof of Theorem \ref{Thm:WeakSol} can be considered as finished.\\

Next we focus on the proof of Theorem \ref{Thm:Main}. Choosing $\xi = -\Delta v_n$ in equation \eqref{NSFinDim}, we obtain that
\begin{equation}\label{NSSecondAE}
	\begin{split}
\frac{1}{2}\frac{\d}{\d t}\|\nabla v_n\|^2_2 &+\|\Delta v_n\|^2_2\\
 &=I^n_1+I^n_2+I^n_3+\int_\Omega (\Delta v_n\cdot\nabla) M_n\cdot\Delta M_n,
	\end{split}
\end{equation}
where 
\begin{align*}
	I^n_1&=\int_\Omega (v_n\cdot\nabla)v_n\cdot\Delta v_n,\\
	I^n_2&=-\int_\Omega \nabla^\top H_\ext M_n  \cdot\Delta v_n, \\
	I^n_3&=\int_\Omega (W'(F_n) (F_n)^\top)\cdot\nabla\Delta v_n.
\end{align*}
Notice that we used the identity $\dvr(\nabla M_n\odot\nabla M_n)=\nabla\frac{|\nabla M_n|^2}{2}+(\nabla M_n)^T\Delta M_n$ and the solenoidality of $\Delta v_n$ in the third term of \eqref{NSFinDim}  to infer 
\begin{equation*}
\int_\Omega\dvr(\nabla M_n\odot\nabla M_n)\cdot\Delta v_n=\int_\Omega(\nabla M_n)^T\Delta M_n\cdot\Delta v_n=\int_\Omega(\Delta v_n\cdot\nabla) M_n\cdot\Delta M_n.
\end{equation*}
In order to handle the third and the fourth term on the right hand side of \eqref{NSSecondAE}, we work with equations \eqref{DGFinDim} and \eqref{LLGFinDim}. We begin with testing \eqref{DGFinDim} with $-\Delta F_n$ and obtain
\begin{equation}  \label{DGSecondAE}
	\frac{1}{2}\frac{\d}{\dt}\|\nabla F_n\|_2^2+\kappa\|\Delta F_n\|_2^2=\int_\Omega (v_n\cdot\nabla) F_n \cdot\Delta F_n-\int_\Omega (\nabla v_nF_n)\cdot\Delta F_n = I_4^n + I_5^n
\end{equation}
with
\begin{align*}
	I^n_4&=\int_\Omega (v_n\cdot\nabla) F_n \cdot\Delta F_n,\\
	I^n_5&=-\int_\Omega\left(\nabla v_n F_n\right) \cdot \Delta  F_n.
\end{align*}
Next we take the gradient of \eqref{LLGFinDim} and test the resulting equality with $-\nabla \Delta M_n$ to obtain
\begin{align*}
	\frac{1}{2}\frac{\d}{\d t}&\|\Delta M_n\|^2_2 +\|\nabla\Delta M_n\|^2_2\\
	&=\int_\Omega \nabla\Big(-(v_n\cdot\nabla) M_n+|\nabla M_n|^2M_n-M_n\times(\Delta M_n+H_\ext - (M_n\cdot e_3)e_3)
	\\ &-(M_n\cdot H _\ext )M_n+H_\ext + ((M_n\cdot e_3)^2M_n- (M_n\cdot e_3)e_3)\Big)\cdot(-\nabla\Delta M_n).
\end{align*}
We integrate by parts in the first term on the right hand side of the latter equality to get
\begin{equation}\label{LLGSecondAE}
	\begin{split}
	\frac{1}{2}\frac{\d}{\d t}\|\Delta M_n\|^2_2+\|\nabla\Delta M_n\|^2_2
	=&-\int_\Omega (\Delta v_n\cdot\nabla) M_n\cdot\Delta M_n+\sum_{i=6}^{16}I^n_i,
	\end{split}
\end{equation}	
where 
	\begin{align*}
		I^n_6&=-\sum_{k=1}^d\int_\Omega (\partial_k v_n\cdot\nabla) \partial_k M_n\cdot\Delta M_n,\\
		I^n_7&=-\int_\Omega (v_n\cdot\nabla) \Delta M_n\cdot\Delta M_n,\\
		I^n_8&=-\int_\Omega 2\Big((\nabla^2 M_n\nabla M_n)\otimes M_n\Big)\cdot\nabla\Delta M_n,\\
		I^n_{9}&=-\int_\Omega |\nabla M_n|^2\nabla M_n\cdot\nabla\Delta M_n,\\
		I^n_{10}&=\int_\Omega \Big(\nabla M_n\times\big(\Delta M_n+H_\ext - (M_n \cdot e_3)e_3\big)\Big)\cdot\nabla\Delta M_n,\\
		I^n_{11}&=\int_\Omega (M_n\times\nabla H_\ext)\cdot\nabla\Delta M_n,\\
		I^n_{12}&=-\int_\Omega \big(M_n \times \nabla (M_n \cdot e_3)e_3\big) \cdot \nabla \Delta M_n,\\
		I^n_{13}&=\int_\Omega (M_n\cdot H_\ext )\cdot(\nabla M_n\cdot\nabla\Delta M_n),\\
		I^n_{14}&=\int_\Omega \Big((\nabla\Delta M_n)^\top M_n\Big)\cdot\Big((\nabla M_n)^\top H_\ext+(\nabla H_\ext)^\top 								M_n\Big),\\
		I^n_{15}&=\int_\Omega \nabla H_\ext \cdot\nabla \Delta M_n,\\
		I^n_{16}&=\int_\Omega \nabla \Big((M_n\cdot e_3)^2M_n- (M_n\cdot e_3)e_3\Big) \cdot \nabla \Delta M_n.
	\end{align*}
Summing up \eqref{NSSecondAE}, \eqref{DGSecondAE} and \eqref{LLGSecondAE}, we obtain
\begin{equation}\label{TotalSum}
	\begin{split}
	\frac{1}{2}\frac{\d}{\d t} & \big(\|\nabla v_n\|_2^2+\|\nabla F_n\|_2^2+\|\Delta M_n\|^2_2\big)+\|\Delta v_n\|_2^2+\|\Delta F_n\|_2^2+\|\nabla\Delta M_n\|^2_2\\
	& = \sum_{i=1}^{16} I^n_i.
	\end{split}
\end{equation}
We estimate each $I^n_i$ separately. For the estimate of $I^n_1$ we employ Agmon's inequality having the form 
\begin{equation*}
	\|u\|_\infty\leq c\|u\|^\frac{1}{2}_{1,2}\|u\|^\frac{1}{2}_{2,2} \qquad \text{ for all }u\in W^{2,2}(\Omega)^2
\end{equation*}
and \eqref{W22Est} in the appendix to obtain
\begin{align*}
	|I^n_1|& \leq \|v_n\|_{\infty}\|\nabla v_n\|_{2}\|\Delta v_n\|_2\leq c\left(\|v_n\|_{1,2}\|v_n\|_{2,2}\right)^\frac{1}{2}\|\nabla v_n\|_2\|\Delta v_n\|_2 \\
	&\leq c \|\nabla v_n\|_2^\frac{3}{2}\|\Delta v_n\|_2^\frac{3}{2}
	\leq c\|\nabla v_n\|_2^6+\delta\|\Delta v_n\|_2^2
\end{align*}
using the Young inequality at the end,  as usual for some $\delta >0$ small, specified later. The following terms are estimated applying the Young inequality iteratively. For the estimates of $L^4$--norms of $v_n$ and $F_n$, Ladyzhenskaya's inequality \eqref{Lad2D} is applied.  Agmon's inequality is also applied in the estimation of $I^n_3$ 
\begin{align*}
	|I^n_3| & \leq \left| \int_\Omega \dvr(W'(F_n)(F_n)^\top ) \cdot \Delta v_n \right| \leq \int_\Omega (|W''(F_n)||F_n| +
	 |W'(F_n)| )|\nabla F_n | \Delta v_n| \\
	 & \leq c \|F_n \|_\infty \|\nabla F_n\|_2 \|\Delta v_n\|_2 \leq c ( \|\nabla F_n\|_2 \|\Delta F_n\|_2)^\frac12 \|\nabla F_n\|_2
	 \| \Delta v_n \|_2 \\
	 &  \leq c\|\nabla F_n\|_2^6 +\delta \|\Delta v_n\|_2^\frac43 \|\Delta F_n\|_2^\frac23 \leq c\|\nabla F_n\|_2^6 +
	 \delta ( \|\Delta v_n\|_2^2 +  \|\Delta F_n\|_2^2),
\end{align*}
where we made use of the assumptions \eqref{W2} and \eqref{W3}. Moreover, we obtain employing \eqref{GenLad2D} and \eqref{W22Est} for the interpolation of $L^4$--norms of $\nabla v_n$ and $\nabla F_n$ that
\begin{align*}
	|I^n_2|\leq& c\|\nabla H_\ext \|_2^2+\delta \|\Delta v_n\|_2^2,\\
	|I^n_4|\leq& \|v_n\|_{4}\|\nabla F_n\|_4\|\Delta F_n\|_2 \\
	\leq & c\|v_n\|^\frac{1}{2}_2\|\nabla v_n\|^\frac{1}{2}_2\|\nabla F_n\|^\frac{1}{2}_2\left(\|\nabla F_n\|^2_2+\|\Delta F_n\|^2_2\right)^\frac{1}{4}\|\Delta F_n\|_2\\
	\leq& c \|v_n\|_2\|\nabla v_n\|_2\|\nabla F_n\|_2(\|F_n\|_2+\|\Delta F_n\|_2)+\delta\|\Delta F_n\|_2^2\\
	\leq& c(\|v_n\|_2\|\nabla v_n\|_2\|\nabla F_n\|_2\|F_n\|_2+\|v_n\|^2_2\|\nabla v_n\|^2_2\|\nabla F_n\|^2_2)+2\delta\|\Delta F_n\|_2^2\\
	\leq& c(\|v_n\|_2\|F_n\|_2\left(\|\nabla v_n\|_2^2+\|\nabla F_n\|^2_2\right)+\|v_n\|^2_2\left(\|\nabla v_n\|^4_2+\|\nabla F_n\|^4_2\right) )+2\delta\|\Delta F_n\|_2^2,\\
	|I^n_5|\leq& \|\nabla v_n\|_4\|\Delta F_n\|_2\|F_n\|_4 \\
	\leq & c\|\nabla v_n\|_2(\|\nabla v_n\|_2+\|\Delta v_n\|_2)\|F_n\|_2\|\nabla F_n\|_2+\delta\|\Delta F_n\|_2^2,\\
	\leq& c\|F_n\|_2(\|\nabla v_n\|^4_2+\|\nabla F_n\|^2_2)+c\|F_n\|_2^2(\|\nabla v_n\|_2^4+\|\nabla F_n\|_2^4)\\
	&+\delta\left(\|\Delta F_n\|_2^2+\|\Delta v_n\|_2^2\right).
	\end{align*}
The following estimates of $I^n_6$ to $I^n_7$ are related to the Landau-Lifshitz-Gilbert equation \eqref{MagnetSystem}. Proceeding as above, we apply additionally \eqref{LadForLaplacian} for the estimate of the $L^4$--norm of $\Delta M_n$.
	\begin{align*}
	|I^n_6|\leq& c\|\nabla v_n\|_4\|\nabla^2M_n\|_2\|\Delta M_n\|_4 \\
	\leq& c\|\nabla v_n\|^\frac{1}{2}_2\left(\|\nabla v_n\|^2_2+\|\Delta v_n\|_2^2\right)^\frac{1}{4}\|\Delta M_n\|_2^\frac{3}{2}\left(\|\Delta M_n\|^2_2+\|\nabla \Delta M_n\|^2_2\right)^\frac{1}{4}\\
	\leq& c\|\nabla v_n\|_2\left(\|\nabla v_n\|_2+\|\Delta v_n\|_2\right)+c\|\Delta M_n\|^3_2\left(\|\Delta M_n\|_2+\|\nabla \Delta M_n\|_2\right)\\
	\leq& c\left(\|\nabla v_n\|_2^2+\|\Delta M_n\|_2^4++\|\Delta M_n\|_2^6\right)+\delta\left(\|\Delta v_n\|^2_2+\|\nabla\Delta M_n\|^2_2\right),\\
	|I^n_7|\leq& \|v_n\|_\infty\|\nabla\Delta M_n\|_2\|\Delta M_n\|_2 \\ 
	\leq & c\|v_n\|^\frac{1}{2}_{1,2}\|v_n\|^\frac{1}{2}_{2,2}\|\Delta M_n\|_2\|\nabla\Delta M_n\|_2\\
	\leq& c\|\nabla v_n\|_2\|\Delta v_n\|_2\|\Delta M_n\|_2^2+\delta\|\nabla\Delta M_n\|_2^2 \\
	\leq & c \left(\|\nabla v_n\|_2^2+\|\Delta M_n\|_2^4\right)+\delta\left(\|\Delta v_n\|_2^2+\|\nabla\Delta M_n\|_2^2\right).
	\end{align*}
Terms $I^n_8$ and $I^n_{9}$ consist of the critical term on the rhs of \eqref{MagnetSystem}. Employing $|M_n|=1$ a.e.\ in $(0,t^*)\times\Omega$ along with \eqref{AgmonVariant} and \eqref{L6Interp}, we have
	\begin{align*}
	|I^n_8|\leq& c\|\nabla^2 M_n\|_2\|\nabla M_n\|_\infty\|M_n\|_\infty\|\nabla\Delta M_n\|_2\\
	\leq& c\|\Delta M_n\|_2\|\nabla M_n\|_2^\frac{1}{2}\left(\|\nabla M_n\|_2^2 +\|\nabla \Delta M_n\|_2^2\right)^\frac{1}{4}\|\nabla\Delta M_n\|_2\\
	\leq& c\|\Delta M_n\|^2_2\|\nabla M_n\|_2\big(\|\nabla M_n\|_2+\|\nabla \Delta M_n\|_2\big)+\delta\|\nabla\Delta M_n\|^2_2\\
	\leq& c\big(\|\nabla M_n\|_2^2\|\Delta M_n\|_2^2+\|\nabla M_n\|^2_2\|\Delta M_n\|_2^4\big)+2\delta\|\nabla\Delta M_n\|^2_2,\\
	|I^n_{9}|\leq& \|\nabla M_n\|^3_6\|\nabla\Delta M_n\|_2\leq c\|\nabla M_n\|_2\big(\|\nabla M_n\|_2^2+\|\Delta M_n\|_2^2\big)\|\nabla\Delta M_n\|_2\\
	\leq& c\big(\|\nabla M_n\|_2^6+\|\nabla M_n\|_2^2\|\Delta M_n\|^4_2\big)+\delta\|\nabla\Delta M_n\|^2_2.
	\end{align*}
From the remaining terms, only $I^n_{10}$ involves higher order derivatives
	\begin{align*}
	|I^n_{10}|\leq& \|\nabla M_n\|_\infty\left(\|\Delta M_n\|_2+\|H_\ext \|_2 + |\Omega|^\frac{1}{2}\right)\|\nabla\Delta M_n\|_2\\
	\leq& c\|\nabla M_n\|_2^\frac{1}{2}\left(\|\nabla M_n\|_2^2+\|\nabla \Delta M_n\|_2^2\right)^\frac{1}{4}\left(\|\Delta M_n\|_2+\|H_\ext \|_2 + 1\right)\|\nabla\Delta M_n\|_2\\
	\leq& c\|\nabla M_n\|_2\left(\|\nabla M_n\|_2+\|\nabla\Delta M_n\|_2\right)\left(\|\Delta M_n\|_2^2+\|H_\ext \|^2_{2} +1\right)+\delta\|\nabla\Delta M_n\|_2^2\\
	\leq& c\left(\|\nabla M_n\|_2^2\|\Delta M_n\|_2^2+\|\nabla M_n\|_2^2(\|H_\ext \|^2_{2}+1)\right.\\
	&\left.+\|\nabla M_n\|_2^2\|\Delta M_n\|_2^4+\|\nabla M_n\|_2^2(\|H_\ext\|_{2}^4 +1)\right)+2\delta\|\nabla\Delta M_n\|_2^2.
	\end{align*}
The other terms are straightforward (it is always $|M_n| \equiv 1$) with
\begin{align*}
	|I^n_{11}|\leq& c\|H_\ext \|_{1,2}^2+\delta\|\nabla\Delta M_n\|_2^2,\\
	|I^n_{12}| \leq & c \|\nabla M_n\|_2^2 + \delta \|\nabla \Delta M_n\|_2^2,\\
	|I^n_{13}|\leq&
	c\|H_\ext\|_2^2\left(\|\nabla M_n\|_2^2 +\|\Delta M_n\|_2^2\right)+2\delta\|\nabla\Delta M_n\|_2^2,\\
	|I^n_{14}|\leq&
	 c\|H_\ext\|_{1,2}^2\left(1+\|\nabla M_n\|_2^2+\|\Delta M_n\|_2^2\right)+ 2\delta\|\nabla\Delta M_n\|_2^2,\\
	|I^n_{15}|\leq& c\|H_\ext \|_{1,2}^2+\delta\|\nabla\Delta M_n\|_2^2, \\
	|I^n_{16}| \leq & c\|\nabla M_n\|_2^2 + \delta \| \nabla \Delta M_n \|_2^2.
\end{align*}
We integrate \eqref{TotalSum} over $(0,t)$, add \eqref{InterMedIneq} and make use of the estimates for $I^n_1$ to $I^n_{16}$ from above. An iterative application of Young's inequality  and a suitably small choice for $\delta$ then yield 
\begin{equation*}
\begin{split}
	E(t)& + \frac12 \big(\|\nabla F_n(t)\|_2^2+\|\nabla M_n(t)\|_2^2+\|\Delta M_n(t)\|_2^2 \big)\\
	&+c\int_0^t\|\nabla v_n(s)\|_2^2+\|\nabla F_n(s)\|_2^2+\|\Delta v_n(s)\|_2^2+\|\Delta F_n(s)\|_2^2+\|\nabla \Delta M_n(s)\|_2^2\ds\\&\leq E(0) +c\int_0^t1+\left(\|\nabla v_n(s)\|_2^2+\|\nabla F_n(s)\|^2_2+\|\Delta M_n(s)\|^2_2\right)^3\ds
	\end{split}
\end{equation*}
with $c$ depending on the bounds of $\{(v_n,F_n,M_n)\}_{n=1}^\infty$ from \eqref{InterMedIneq}.
It is not difficult to realize that the solution of the initial value problem
\begin{align*}
	z_n'(t)&=c\left(1+z_n^3(t)\right),\\
	z_n(0)&=\|v_0\|_{1,2}^2+\|F_0\|_{1,2}^2+\|\nabla M_0\|_{1,2}^2
\end{align*}
exists on $(0,T^\ast)$ for some $T^\ast\in(0,T]$. With the help of Lemma~\ref{Lem:SolODEComp} we deduce that
\begin{equation*}
	\begin{split}
	E(t)&+ \frac12 \big(\|\nabla v_n(t)\|_2^2+\|\nabla F_n(t)\|_2^2+\|\Delta M_n(t)\|_2^2\big)\\&+\int_0^t\|\nabla v_n(s)\|_2^2+\|\nabla F_n(s)\|_2^2+\|\Delta M_n(s)\|_2^2 \ds\\
	&+\int_0^t \|\Delta v_n(s)\|_2^2+\|\Delta F_n(s)\|_2^2+\|\nabla \Delta M_n(s)\|_2^2\ds\leq c
	\end{split}
\end{equation*}
for all $t \in (0,T^*)$.
Therefore, by \eqref{InterMedIneq}, we conclude
\begin{align*}
	\|v_n\|_{L^\infty(0,T^\ast;W^{1,2}(\Omega))}&\leq c,\\
	\|v_n\|_{L^2(0,T^\ast;W^{2,2}(\Omega))}&\leq c,\\
	\|F_n\|_{L^\infty(0,T^\ast;W^{1,2}(\Omega))}&\leq c,\\
	\|F_n\|_{L^2(0,T^\ast;W^{2,2}(\Omega))}&\leq c,\\
	\|M_n\|_{L^\infty(0,T^\ast;W^{2,2}(\Omega))}&\leq c,\\
	\|M_n\|_{L^2(0,T^\ast;W^{3,2}(\Omega))}&\leq c.
\end{align*}
Hence, there exists an (not explicitly labeled) subsequence $\{(v_n,F_n,M_n)\}_{n=1}^\infty$ such that  
\begin{alignat*}{5}
	v_n&\rightharpoonup^\ast v&&\text{ in } L^\infty(0,T^\ast;W^{1,2}(\Omega)^2),\ 
	&v_n&\rightharpoonup v&&\text{ in }L^2(0,T^\ast;W^{2,2}(\Omega)^2),\\
	F_n&\rightharpoonup^\ast F&&\text{ in }L^\infty(0,T^\ast;W^{1,2}(\Omega)^{2\times 2}),\
	&F_n&\rightharpoonup F&&\text{ in } L^2(0,T^\ast;W^{2,2}(\Omega)^{2\times 2}),\\
	M_n&\rightharpoonup^\ast M&&\text{ in }L^\infty(0,T^\ast;W^{2,2}(\Omega)^3),\
	&M_n&\rightharpoonup M&&\text{ in } L^2(0,T^\ast;W^{3,2}(\Omega)^3).
\end{alignat*}
We note that the above estimates and the estimates on the time derivatives of $v_n$, $F_n$, $M_n$ and $\nabla M_n$ provide strong convergences that are necessary to verify that $(v,F,M)$ satisfies the weak formulation \eqref{WForm}, see Steps 3 and 5 of the proof of \cite[Theorem 3.2]{BeFoLiSc}. Thanks to the regularity of the limit functions $v$, $F$, $M$, we can integrate by parts in space and time to get
\begin{equation}\label{SForm}
	\begin{split}
	&\int_0^{T^*}\int_\Omega (\tder v+ (v\cdot\nabla)v+\dvr\left(\nabla M\odot\nabla M-W'(F)F^\top-\nabla v\right) \\
	 &\qquad \qquad -(\nabla H_\ext^\top M))\cdot\phi\dx\dt=0,\\
	&\int_0^{T^*}\int_\Omega (\tder F+(v\cdot\nabla)F-(\nabla vF)-\kappa \Delta F)\cdot\xi\dx\dt=0,\\
	&\int_0^{T^*}\int_\Omega (\tder M+(v\cdot\nabla M)+ M\times H_\eff -|\nabla M|^2M \\
	& \qquad \qquad -\Delta M + M \times (M \times H_\res) )\cdot\theta\dx\dt=0
	\end{split}
\end{equation}
for all $\phi\in L^2(0,T;W^{1,2}_{0,\dvr}(\Omega)^2)$, $\xi\in L^2(0,T;W^{1,2}_0(\Omega)^{2\times 2})$ and $\theta\in L^2(0,T; L^2(\Omega)^3)$. Obviously, from \eqref{SForm}$_{2,3}$ it follows that equations \eqref{DefGradSystem} and \eqref{MagnetSystem} are satisfied a.e.\ in $(0,T^*)\times\Omega$. 

It remains to show the existence of an associated pressure from \eqref{SForm}. Obviously, the regularity of $v,F,M$ and $H_\ext$ implies that
\begin{equation}\label{PressEq}
\begin{split}
	G(s):=&
	\tder v(s)+(v(s)\cdot\nabla)v(s)\\
	&+\dvr\left(\nabla M(s)\odot\nabla M(s)-W'(F(s))(F(s))^\top-\nabla v(s)\right)\\&-(\nabla H_\ext^\top(s)M(s))\in L^2(\Omega)^2
	\end{split}
\end{equation}
for a.a.\ $s \in [0,T^*]$.
Moreover, \eqref{SForm}$_1$ and the Helmholtz-Weyl decomposition imply that $G(s)=\nabla \tilde p(s)$ for some $\tilde p(s)\in W^{1,2}(\Omega)$ a.e.\ in $[0,T^*]$. Consequently, we have $\|\nabla \tilde p\|_{L^2(0,T^*;L^2(\Omega))}=\|G\|_{L^2(0,T^*;L^2(\Omega))}$. Defining $p=\tilde p-\frac{1}{|\Omega|}\int_\Omega \tilde p$, we obtain $\|p\|_{L^2(0,T^*;L^2(\Omega))}\leq c\|\nabla p\|_{L^2(0,T^*;L^2(\Omega))}$ by the Poincar\'e inequality. Therefore we have shown the existence of an associated pressure and from \eqref{PressEq} we conclude that \eqref{NSSystem} is fulfilled a.e.\ in $(0,T^*)\times\Omega$. \hfill $\Box$
\vspace{1em}

The key ingredient for the existence proofs above is the following lemma.
\begin{lemma}\label{KeyLemma}
	Let the assumptions of Theorem \ref{Thm:Main} be satisfied and let $v\in V^n(t_0)$. Then there is $t^*\in(0,t_0)$ depending on $n,L,t_0,F_0,M_0$ and $H_\ext$ such that we can find a unique pair $(F,M)$ possessing the regularity
	\begin{equation*}
	\begin{split}
	&F\in L^\infty(0,t^*;W^{1,2}_0(\Omega)^{2\times 2})\cap L^2(0,t^*;W^{2,2}(\Omega)^{2\times 2}),\ \tder F\in L^2(0,t^*;L^2(\Omega)^{2\times 2}),\\
	&M\in W^{1,\infty}(0,t^*;L^2(\Omega)^3)\cap W^{1,2}(0,t^*;W^{1,2}(\Omega)^3)\cap L^2(0,t^*;W^{3,2}(\Omega)^3)
	\end{split}
	\end{equation*}
	satisfying \eqref{DefGradSystem} and \eqref{MagnetSystem} a.e.\ in $(0,t^*)\times\Omega$, i.e., 
	\begin{equation}\label{FMEquations}
		\begin{split}
		&\tder F+(v\cdot\nabla) F=\nabla vF+\kappa\Delta F,\\
		&\tder M+(v\cdot\nabla) M=|\nabla M|^2M+\Delta M-M\times H_\eff-M\times (M \times
		H_\res)
		\end{split}
	\end{equation}
	together with the boundary and initial conditions from \eqref{BoundInitCond}. Moreover, $|M|=1$ a.e.\ in $(0,t^*)\times\Omega$.
	\begin{proof}
		The main steps of the proof follow the proof of \cite[Lemma 4.2]{BeFoLiSc}, which combines the Galerkin method and standard ODE theory for the existence of an approximate solution. Here, we outline the differences to that proof only.
		
		Compared to the corresponding Landau-Lifshitz-Gilbert equation in \cite{BeFoLiSc}, we here have an additional term that is linear in $M$. This allows to conclude from \cite[Lemma 4.2]{BeFoLiSc} that
		$$ M\in W^{1,\infty}(0,t^*;L^2(\Omega)^3)\cap  L^2(0,t^*;W^{3,2}(\Omega)^3).$$
		By \eqref{FMEquations}, we thus obtain that $M\in  W^{1,2}(0,t^*;W^{1,2}(\Omega)^3)$. Hence the asserted regularity of $M$ follows.
		
		It remains to show the improved regularity of $F$. As in \cite{BeFoLiSc}, we consider a smooth approximation $F_k$ of $F$, which satisfies \cite[Equality (86)]{BeFoLiSc}, i.e.,
		\begin{equation}\label{FForGalApprox}
			\int_\Omega \left(\tder F_k+(v\cdot\nabla)F_k -\nabla vF_k\right)\cdot\Xi +\kappa\nabla F_k\cdot\nabla\Xi=0 \text{ in }(0,t^*)
		\end{equation}
		for all $\Xi\in Y^k$, where $Y^k$ is the span of the first $k$ functions that form an orthonormal basis in $L^2(\Omega)^{2\times 2}$ and an orthogonal basis of $W^{1,2}_0(\Omega)^{2\times 2}$ consisting of eigenfunctions of the Laplace operator with homogeneous Dirichlet boundary conditions. The apriori estimate 
		\begin{equation}\label{FirstAEFK}
			\|F_k\|_{L^\infty(0,t^*;L^2(\Omega))}+\|F_k\|_{L^2(0,t^*;W^{1,2}(\Omega))}\leq c
		\end{equation}
		is then derived in a standard way, cf.\ the proof of \cite[Equation (88)]{BeFoLiSc}. Setting $\Xi=-\Delta F_k$ in \eqref{FForGalApprox} and integrating over $(0,t)\subset(0,t^*)$, we get 
		\begin{equation*}
			\frac12 \|\nabla F_k(t)\|_2^2+\kappa\|\Delta F_k\|^2_{L^2(0,t;L^2(\Omega))}=\frac12 \|\nabla F_k(0)\|_2^2+\int_0^t\int_\Omega(v\cdot\nabla)F_k\cdot\Delta F_k -\nabla vF_k\cdot\Delta F_k.
		\end{equation*}
		Employing the regularity of $v$, and the H\"older and Young inequalities, we have
		\begin{align*}
			\frac12 &\|\nabla F_k(t)\|_2^2 +\kappa\|\Delta F_k\|^2_{L^2(0,t;L^2(\Omega))} \\
			&\leq \|F_k(0)\|^2_{W^{1,2}(\Omega)}+c\|v\|^2_{C([0,t_0];C^1(\overline{\Omega}))}\|F_k\|^2_{L^2(0,t^*;W^{1,2}(\Omega))}+\frac{\kappa}{2}\|\Delta F_k\|^2_{L^2(0,t;L^2(\Omega))}.
		\end{align*}
		Similarly, setting $\Xi=\tder F_k$ in \eqref{FForGalApprox} and integrating over $(0,t)\subset(0,t^*)$, we obtain
		\begin{align*}
		\|\tder & F_k\|^2_{L^2(0,t;L^2(\Omega))} +\frac{\kappa}{2}\|\nabla F_k(t)\|_2^2\\
		 &=\frac{\kappa}{2}\|\nabla F_k(0)\|_2^2-\int_0^t\int_\Omega(v\cdot\nabla)F_k\cdot\tder F_k -\nabla v F_k\cdot\tder F_k
		\end{align*}
		and 
		\begin{align*}
		\|& \tder F_k\|^2_{L^2(0,t;L^2(\Omega))}+\frac{\kappa}{2}\|\nabla F_k(t)\|_2^2 \\
		& \leq \frac{\kappa}{2}\|F_k(0)\|^2_{W^{1,2}(\Omega)}+c\|v\|^2_{C([0,t_0];C^1(\overline{\Omega}))}\|F_k\|^2_{L^2(0,t^*;W^{1,2}(\Omega))}+\frac{1}{2}\|\tder F_k\|^2_{L^2(0,t;L^2(\Omega))}.
		\end{align*} 
		Therefore, additionally to estimate \eqref{FirstAEFK}, we have
		\begin{equation*}
			\|F_k\|_{L^2(0,t^*;W^{2,2}(\Omega))}+\|F_k\|_{L^\infty(0,t^*;W^{1,2}_0(\Omega))}+\|\tder F_k\|_{L^2(0,t^*;L^2(\Omega))}\leq c.
		\end{equation*}
		Hence, there exists a (not explicitly labeled) subsequence of $(F_k)$ such that 
		\begin{equation*}
			\begin{split}
				F_k &\rightharpoonup F\text{ in }L^2(0,t^*;W^{2,2}(\Omega)),\\
				F_k &\rightharpoonup^* F\text{ in }L^\infty(0,t^*;W^{1,2}(\Omega)),\\
				\tder F_k &\rightharpoonup \tder F\text{ in }L^2(0,t^*;L^2(\Omega)).
			\end{split}
		\end{equation*}
	Finally, multiplying \eqref{FForGalApprox} by a smooth compactly supported function in $(0,t^*)$, we can employs the above convergences in the limit $k\to \infty$. Moreover, we use the density of $\bigcup_{k=1}^\infty Y^k$ in $W^{1,2}_0(\Omega)^{2\times 2}$ and perform an integration by parts in the term with $\kappa$ to obtain \eqref{FMEquations}$_1$.
	\end{proof}
\end{lemma}

\section{Proof of Theorem \ref{Thm: WSUniqueness}} \label{sec:uni} 

In order to prove the weak-strong uniqueness, we need to be able to choose components of a weak solution to \eqref{NSSystem}--\eqref{MagnetSystem} as corresponding test functions in the weak formulation \eqref{WForm}. At first, this is not possible since the $F$ component of a weak solution does in general not belong to $L^4\big(0,T;W^{1,2}_0(\Omega)^{2\times 2}\big)$, which is the predual of the space that $\tder F$ belongs to, see Step 3 of the proof of \cite[Theorem 3.2]{BeFoLiSc} for details. Therefore, the duality pairing $\int_0^T\left\langle \tder F, F\right\rangle$ is not defined. This is overcome by the following lemma.
\begin{lemma}\label{Lem:TDer}
	Let $(v,F,M)$ be a weak solution to \eqref{NSSystem}--\eqref{MagnetSystem}. Then it follows that 
	\begin{equation}\label{TDerEst}
	\tder F\in \left(L^2\big(0,T;W^{1,2}_0(\Omega)^{2\times2}\big)\cap L^4\left(0,T;L^4(\Omega)^{2\times 2}\right)\right)^*.
	\end{equation}

	\begin{proof}
If $\xi \in C^\infty_c((0,T)\times\Omega)^{2\times 2}$ we have for the distribution $\tder F$ 
\begin{equation*}
\begin{split}
	\left\langle \tder F,\xi \right\rangle=\int_0^{T}\int_\Omega-(v\cdot\nabla)F\cdot\xi+\nabla vF\cdot\xi-\kappa \nabla F\cdot\nabla\xi \dx\ds.
\end{split}
\end{equation*}
We can perform the integration by parts in the first term on the right hand side and, using the solenoidality of $v$ and the fact that the trace of $v$ is zero on $\partial\Omega$, we arrive at
\begin{equation*}
	-\int_0^{T}\int_\Omega(v\cdot\nabla)F\cdot\xi = - \int_0^T \int_\Omega \dvr (v \otimes
	F)\cdot \xi = \int_0^T \int_\Omega v \otimes F \cdot \nabla \xi.
\end{equation*}
By the H\"older inequality and the 2-D version of the Ladyzhenskaya inequality \eqref{Lad2D}, we obtain that
\begin{equation*}
\begin{split}
	|\left\langle \tder F,\xi \right\rangle|\leq& \int_0^{T}\|v\|_{4}\|F\|_4\|\nabla\xi\|_2+\|\nabla v\|_2\|F\|_4\|\xi\|_4+\kappa \|\nabla F\|_2\|\nabla\xi\|_2\ds\\
	\leq& \int_0^{T} \left(\|v\|^\frac{1}{2}_{2}\|\nabla v\|^\frac{1}{2}_2\|F\|_2^\frac{1}{2}\|\nabla F\|^\frac{1}{2}_2+\kappa \|\nabla F\|_2\right)\|\nabla \xi\|_2\\
	&+c\|\nabla v\|_{2}\|F\|_2^\frac{1}{2}\|\nabla F\|_2^\frac{1}{2}\|\xi\|_4 \ds \\
	\leq& \Bigl(c\|v\|^\frac{1}{2}_{L^\infty(0,T;L^2(\Omega))}\|F\|^\frac{1}{2}_{L^\infty(0,T;L^2(\Omega))}\|\nabla v\|^\frac{1}{2}_{L^2(0,T;L^2(\Omega))}\|\nabla F\|^\frac{1}{2}_{L^2(0,T;L^2(\Omega))}\\
	&+\kappa\|\nabla F\|_{L^2(0,T;L^2(\Omega))}\Bigr)\|\xi\|_{L^2(0,T;W^{1,2}(\Omega))}\\
	&+c\|F\|^\frac{1}{2}_{L^\infty(0,T;L^2(\Omega))}\|\nabla v\|_{L^2(0,T;L^2(\Omega))}\|\nabla F\|^\frac{1}{2}_{L^2(0,T;L^2(\Omega))}\|\xi\|_{L^4(0,T;L^4(\Omega))}\\
	\leq& c (\|\xi\|_{L^4(0,T;L^4(\Omega))}+\|\xi\|_{L^2(0,T;W^{1,2}(\Omega))}).
\end{split}
\end{equation*}
Using the fact that $ C^\infty_c((0,T)\times\Omega)^{2\times 2}$ is dense in the spaces $L^2(0,T;W^{1,2}_0(\Omega)^{2\times2})$ and $ L^4(0,T;L^4(\Omega)^{2\times 2})$, we conclude \eqref{TDerEst}.
	\end{proof}	
\end{lemma}
 
Let us assume that there is a strong solution $(v^1,F^1,M^1)$ and a weak solution $(v^2,F^2,M^2)$ to system \eqref{NSSystem}--\eqref{MagnetSystem} that exist on $(0,T)$ and are equipped with the same initial conditions. In what follows we denote the system for the corresponding solution with the corresponding superscript. By a density argument we are able to test the weak formulation of the  difference \eqref{NSSystem}$^1-$\eqref{NSSystem}$^2$ with $v:=v^1-v^2$ on $(0,t)\times\Omega$, $t\in (0,T)$, to obtain
\begin{equation}\label{NSEnergy}
	\frac{1}{2}\|v(t)\|_{2}^2+\int_0^t\|\nabla v\|^2_{2}=\sum_{i=1}^3I_i,
\end{equation}
where we denote
\begin{align*}
	I_1&:=-\int_0^t\int_\Omega \left(\left(v^1\cdot\nabla\right)v^1-\left(v^2\cdot\nabla\right)v^2\right)\cdot\left(v^1-v^2\right),\\
	I_2&:=-\int_0^t\int_\Omega\dvr \left(\nabla M^1\odot\nabla M^1-\nabla M^2\odot\nabla M^2\right)\cdot\left(v^1-v^2\right),\\
	I_3&:=\int_0^t\int_\Omega\dvr \left(W'\left(F^1\right)\left(F^1\right)^\top-W'\left(F^2\right)\left(F^2\right)^\top\right)\cdot\left(v^1-v^2\right).
\end{align*}
Note that the term involving the pressure disappeared due to the solenoidality of $v^1$ and $v^2$. In order to handle the term $I_2$, we use equation \eqref{MagnetSystem} for $M^1$ and $M^2$.

We note that $F^2$ is an admissible test function in \eqref{DefGradSystem} since, by interpolation, $F^2$ possesses the regularity required by Lemma \ref{Lem:TDer}. Again using a density argument and testing the difference of \eqref{DefGradSystem} for $F^1$ and $F^2$ with $F:=F^1-F^2$, we obtain
\begin{equation}\label{DefGradDifTest}
\frac{1}{2}\|F(t)\|^2_2+\int_0^t\|\nabla F(s)\|^2_2\d s=\sum_{i=4}^5I_i,
\end{equation}
where 
\begin{align*}
	I_4&=-\int_0^t\int_\Omega ((v^1\cdot\nabla)F^1-(v^2\cdot\nabla)F^2)\cdot F=-\int_0^t\int_\Omega (v\cdot\nabla)F^1\cdot F,\\
	I_5&=\int_0^t\int_\Omega (\nabla v^1F^1-\nabla v^2F^2)\cdot F.
\end{align*}
First, we take the difference of \eqref{MagnetSystem} for $M^1$ and $M^2$ and get
\begin{equation}\label{MagnetEqDiff}
	\begin{split}
		\tder\left(M^1-M^2\right) -&\Delta\left(M^1-M^2\right)\\
		=&\left(v^2\cdot\nabla\right)M^2-\left(v^1\cdot\nabla\right)M^1+|\nabla M^1|^2M^1-|\nabla M^2|^2M^2\\
		& -M^1\times  H_\eff\left(M^1\right)+M^2\times H_\eff\left(M^2\right)\\
		& -\left(M^1\cdot H_\ext \right)M^1+\left(M^2\cdot H_\ext \right)M^2 \\
		&   - \left(M^2\left(M^2 \cdot e_3\right) -M^1\left(M^1 \cdot e_3\right)\right)
	\end{split}
\end{equation}
Testing \eqref{MagnetEqDiff} with $M:=M^1-M^2$ on $(0,t)\times\Omega$, we obtain
\begin{equation}\label{MEqFirstTest}
	\frac{1}{2}\|M(t)\|_{2}^2+\int_0^t\|\nabla M\|_{2}^2=\sum_{i=6}^{10} I_i,
\end{equation}
where 
\begin{align*}
	I_6&=\int_0^t\int_\Omega \left(\left(v^2\cdot\nabla\right)M^2-\left(v^1\cdot\nabla\right)M^1\right)\cdot M=\int_0^t\int_\Omega-\left(v\cdot\nabla\right)M^1\cdot M,\\
	I_7&=\int_0^t\int_\Omega \left(|\nabla M^1|^2M^1-|\nabla M^2|^2M^2\right)\cdot M\\
		&=\int_0^t\int_\Omega |\nabla M^1|^2|M|^2+\left(|\nabla M^1|-|\nabla M^2|\right)M^2\cdot M,\\
	I_8&=\int_0^t\int_\Omega \left(-M^1\times H_\eff\left(M^1\right) +M^2\times H_\eff\left(M^2\right)\right)\cdot M\\
	&=\int_0^t\int_\Omega \left(-M^2\times\Delta M\right)\cdot M + \left(M^2\times \left(M\cdot e_3 \right)e_3\right) \cdot M,\\
	I_9&=\int_0^t\int_\Omega \left(-\left(M^1\cdot H_\ext \right)M^1+\left(M^2\cdot H_\ext \right)M^2\right)\cdot M\\
	&=\int_0^t\int_\Omega -\left(M^1\cdot H_\ext \right)|M|^2-\left(M\cdot H_\ext \right)\left(M^2\cdot M\right),\\
	I_{10} &= \int_0^t \int_\Omega \left(M^1\left(M^1 \cdot e_3\right) -M^2\left(M^2 \cdot e_3\right)\right)\cdot M\\
	&=\int_0^t \int_\Omega \left(M\left(M^1\cdot e_3\right)+M^2\left(M\cdot e_3\right)\right)\cdot M.
\end{align*}
A test of \eqref{MagnetEqDiff} with $-\Delta M$ yields
\begin{equation}\label{MagDefSecTest}	
	\frac{1}{2}\|\nabla M(t)\|^2_{2}+\int_0^t\|\Delta M\|^2_{2}=\sum_{i=11}^{15} I_i,
\end{equation}
where
\begin{align*}
I_{11}&=\int_0^t\int_\Omega \left(\left(v^1\cdot\nabla\right)M^1-\left(v^2\cdot\nabla\right)M^2\right)\cdot \Delta M,\\
I_{12}&=\int_0^t\int_\Omega -\left(|\nabla M^1|^2M^1-|\nabla M^2|^2M^2\right)\cdot \Delta M \\
&=\int_0^t\int_\Omega \left(-|\nabla M^1|^2M-\left(\nabla \left(M^1+M^2\right)\cdot\nabla M\right)M^2\right) \cdot\Delta M,\\
I_{13}&=\int_0^t\int_\Omega\left(M^1\times H_\eff -M^2\times H_\eff\right)\cdot \Delta M \\
&=\int_0^t\int_\Omega \left(M\times\Delta M^1 - M\times \left(M^1\cdot e_3\right)e_3 + M^2 
\left(M \cdot e_3\right)e_3\right)\cdot \Delta M ,\\
I_{14}&=\int_0^t\int_\Omega \left(\left(M^1\cdot H_\ext\right)M^1-\left(M^2\cdot H_\ext\right)M^2\right)\cdot \Delta M\\
&=\int_0^t\int_\Omega\left(\left(M\cdot H_\ext\right)M^1+\left(M^2\cdot H_\ext\right)M\right)\cdot\Delta M,\\
I_{15}&= \int_0^t \int_\Omega -M \cdot \Delta M +\left(M^1\left(M^1 \cdot e_3\right) -M^2\left(M^2 \cdot e_3\right)\right)\cdot \Delta M.
\end{align*}
Denoting 
\begin{align*}
	f(t)&=\frac{1}{2}\big(\|v(t)\|^2_2+\|F(t)\|^2_2+\|M(t)\|^2_2+\|\nabla M(t)\|^2_2\big),\\
	g(t)&=\int_0^t\|\nabla v(s)\|^2_2+\|\nabla F(s)\|^2_2+\|\nabla M(s)\|^2_2+\|\Delta M(s)\|^2_2\d s
\end{align*}
and summing up \eqref{NSEnergy}, \eqref{DefGradDifTest}, \eqref{MEqFirstTest} and \eqref{MagDefSecTest}, we obtain that
\begin{equation}\label{FinSum}
	f(t)+g(t)+\sum_{i=1}^{15}I_i=0.
\end{equation}
The aim now is to estimate the sum as follows
\begin{equation*}
	\left|\sum_{i=1}^{15}I_i\right|\leq \int_0^th(s)f(s)\d s+\varepsilon g(t),
\end{equation*}
where $h$ is a positive and integrable function on $(0,T)$ and $\varepsilon>0$ is supposed to be suitably small.

Some of the terms coincide with the ones in the two-dimensional uniqueness proof of \cite[Section 3]{SchlZab}. We refer to them to shorten the presentation starting with \cite[Equation~(13)]{SchlZab}
\begin{equation*}
	|I_1|\leq c\int_0^t\|v(s)\|_2^2\|\nabla v^1(s)\|_2^2+\delta\int_0^t\|\nabla v(s)\|_2^2.
\end{equation*}
Using the identities
\begin{align*}
	\dvr\left(\nabla M^i\odot \nabla M^i\right)=\frac{1}{2}\nabla|\nabla M^i|^2+\nabla^T M^i\Delta M^i,\quad i=1,2,\\
	\nabla^T M^i\Delta M^j\cdot v^k=(v^k\cdot\nabla)M^i\cdot \Delta M^j,\quad i,j,k=1,2,
\end{align*}
we obtain due to cancellation of some mixed terms in accordance to \cite[Equation~(19)]{SchlZab}
\begin{equation*}
	\begin{split}
	|I_2+I_{11}|\leq \int_0^t\int_\Omega |v||\nabla M||\Delta M^2|+\int_0^t\int_\Omega|v^2||\nabla M||\Delta M|.
	\end{split}
\end{equation*}
The first term on the right hand side of the latter inequality is estimated using the H\"older and Young inequalities in combination with interpolation inequalities \eqref{Lad2D} and \eqref{LadMatrix} 
\begin{align*}
	\int_0^t\int_\Omega& |v||\nabla M||\Delta M^2|\leq \int_0^t\|v(s)\|_4\|\nabla M(s)\|_4\|\Delta M^2(s)\|_2\d s\\
	\leq &c\int_0^t\|v(s)\|_2^\frac{1}{2}\|\nabla v(s)\|_2^\frac{1}{2}\|\nabla M(s)\|_2^\frac{1}{2}\|\Delta M(s)\|_2^\frac{1}{2}\|\Delta M^2(s)\|_2\d s\\
	\leq &c\int_0^t\|v(s)\|_2\|\nabla M(s)\|_2\|\Delta M^2(s)\|_2^2\d s+\delta\int_0^t\|\nabla v(s)\|_2\|\Delta M(s)\|_2\ds\\
	\leq &c\int_0^t\left(\|v(s)\|_2^2+\|\nabla M(s)\|^2_2\right)\|\Delta M^2(s)\|_2^2\d s+\delta\int_0^t\|\nabla v(s)\|_2^2+\|\Delta M(s)\|_2^2\d s.
\end{align*}
Similarly, we get
\begin{align*}
	\int_0^t\int_\Omega& |v^2||\nabla M||\Delta M|\leq \int_0^t \|v^2(s)\|_2^\frac{1}{2}\|\nabla v^2(s)\|_2^\frac{1}{2}\|\nabla M(s)\|_2^\frac{1}{2}\|\Delta M(s)\|_2^\frac{1}{2}\|\Delta M(s)\|_2\d s\\
	\leq &c\int_0^t\|v^2(s)\|_2\|\nabla v^2(s)\|_2\|\nabla M(s)\|_2\|\Delta M(s)\|_2\d s+\delta\int_0^t\|\Delta M(s)\|_2^2\d s\\
	\leq &c\int_0^t\|v^2(s)\|^2_2\|\nabla v^2(s)\|^2_2\|\nabla M(s)\|_2^2 \ds+2\delta\int_0^t\|\Delta M(s)\|_2^2\d s.
\end{align*}
The Lipschitz continuity of $W'$ assumed in \eqref{W2} allows us to estimate
\begin{align*}
	|I_3| &= \left| \int_0^t \int_\Omega \left((W'(F^1)-W'(F^2))(F^1)^\top + W'(F^2)F^\top \right) \cdot
	\nabla v \right| \\
	& \leq c \int_0^t \int_\Omega (|F^1| + |F^2|)|F||\nabla v|,\\
	|I_5|& \leq \int_0^t \int_\Omega |F^1||F||\nabla v| + |\nabla v^2||F|^2,
\end{align*}
which yields 
\begin{align*}
|I_3| + |I_5|  \leq & c \int_0^t  \|F(s)\|_2^2 (\|\nabla F^1(s)\|_2^2 \ds+\|\nabla F^2(s)\|_2^2 + \|\nabla v^2(s)\|_2^2)\ds \\
& + \delta \int_0^t \|\nabla v(s)\|_2^2+ \|\nabla F(s)\|_2^2.
\end{align*}
It was deduced in \cite[Section 3, Equation~(14)]{SchlZab} that
\begin{align*}
	|I_4|&\leq c \int_0^t\|\nabla F_1(s)\|_2^2\left(\|v(s)\|_2^2+\|F(s)\|_2^2\right)\d s+\delta\int_0^t\|\nabla v(s)\|_2^2+\|\nabla F(s)\|_2^2\d s.
\end{align*}
Referencing again \cite[Equation~(15)]{SchlZab}, we have 
\begin{align*}
	|I_6|\leq 
	 c\int_0^t\left(\|v(s)\|^2_{2}+\|M(s)\|^2_{2}\right)\|\nabla M_1(s)\|^2_{2}\d s+\delta\int_0^t\|\nabla v(s)\|^2_{2}+\|\nabla M(s)\|^2_{2}\d s.
\end{align*}
For the more crucial terms, that we discuss next, we make again use of the embedding $W^{2,2}$ to $W^{1,4}$ and the fact that $|M_2|=1$ a.e.\ in $(0,T)\times\Omega$
\begin{align*}
	|I_7|\leq& \int_0^t\int_\Omega |\nabla M^1|^2|M|^2+ \left|(\nabla M^1+\nabla M^2)\cdot(\nabla M^1-\nabla M^2) (M^2\cdot M)\right|\\
	\leq& \int_0^t\int_\Omega |\nabla M^1|^2|M|^2+\left(|\nabla M^1|+|\nabla M^2|\right)|\nabla M||M^2||M|\\
	\leq& c\int_0^t\int_\Omega\left(|\nabla M^1|^2+|\nabla M^2|^2\right)|M|^2+\delta\int_0^t\|\nabla M(s)\|^2_{2}\d s\\
	\leq& c\int_0^t \left(\|\nabla M^1(s)\|^2_{4}+\|\nabla M^2(s)\|^2_{4}\right)\|M(s)\|^2_{4}\d s+\delta\int_0^t\|\nabla M(s)\|^2_{2}\d s\\
	\leq& c\int_0^t \left(\|M^1(s)\|^2_{2,2}+\|M^2(s)\|^2_{2,2}\right)\left(\|M(s)\|^2_{2}+\|M(s)\|_2\|\nabla M(s)\|_2\right)\d s \\
	&+\delta\int_0^t\|\nabla M(s)\|^2_{2}\d s\\
	\leq& c\int_0^t \left(\|M^1(s)\|^2_{2,2}+\|M^2(s)\|^2_{2,2}\right)\left(\|M(s)\|^2_{2}+\|\nabla M(s)\|_2^2\right)\d s \\
	 &+\delta\int_0^t\|\nabla M(s)\|^2_{2}\d s.
\end{align*}
Next we perform an integration by parts in $I_8$ and obtain
\begin{align*}
	|I_8|=& \left|\int_0^t\int_\Omega (\nabla M^2\times \nabla M)\cdot M\right| \\
	\leq& c\int_0^t\int_\Omega |\nabla M^2|^2|M|^2+\delta\int_0^t\|\nabla M(s)\|^2_{2}\d s
	+ \int_0^t \|M(s)\|_2^2 \ds\\
	\leq& c\int_0^t\|\nabla M^2(s)\|^2_{4}\|M(s)\|^2_{4}\d s+\int_0^t \delta \|\nabla M(s)\|^2_{2} + \|M(s)\|_2^2 \d s\\
	\leq& c\int_0^t\|M^2(s)\|^2_{2,2}\left(\|M(s)\|_2^2+\|\nabla M(s)\|_2^2\right)\d s+\int_0^t \delta \|\nabla M(s)\|^2_{2} + \|M(s)\|_2^2 \d s.
	\end{align*}
The terms $I_9$ and $I_{10}$ are handled easily by
	\begin{align*}
	|I_9|\leq& 
	c\int_0^t\| H_\ext(s)\|_2\left(\|M(s)\|^2_{2}+\|\nabla M(s)\|^2_{2}\right)\d s,\\
	|I_{10}| \leq& c\int_0^t \|M(s)\|_2^2 \ds.
\end{align*}
Next we estimate
\begin{equation*}
	|I_{12}|\leq \int_0^t\int_\Omega |\nabla M^1|^2|M||\Delta M|+(|\nabla M^1|+|\nabla M^2|)|\nabla M||M^2||\Delta M|.
\end{equation*}
We handle each term separately. Applying \eqref{GenLad2D}, and the H\"older and Young inequalities, we obtain
	\begin{align*}
	\int_0^t\int_\Omega &|\nabla M^1|^2|M||\Delta M| \\
	&\leq c\int_0^t \|\nabla M^1(s)\|^4_{8}\|M(s)\|^2_{4}\d s+\delta\int_0^t\|\Delta M(s)\|^2_2\d s\\
	&\leq c\int_0^t \|M^1(s)\|^4_{2,2}\left(\|M(s)\|^2_{2}+\|\nabla M(s)\|^2_2\right)\ds+\delta\int_0^t\|\Delta M(s)\|^2_2\d s.
\end{align*}
Similarly, we get
\begin{align*}
\int_0^t& \left(|\nabla M^1|+|\nabla M^2|\right)|\nabla M||M^2||\Delta M|\\
	\leq& \int_0^t \left(\|\nabla M^1(s)\|^2_{4}+\|\nabla M^2(s)\|^2_{4}\right)\|\nabla M(s)\|^2_{4}\d s+\delta\int_0^t\|\Delta M(s)\|^2_{2} \d s\\
	\leq& c\int_0^t \left( \|\nabla M^1(s)\|_{2}\left(\|\nabla M^1(s)\|^2_{2} 
	 + \|\Delta M^1(s)\|^2_{2}\right)^\frac{1}{2} \right. \\
	  & \left.  +\|\nabla M^2(s)\|_{2}\left(\|\nabla M^2(s)\|^2_{2}+\|\Delta M^2(s)\|^2_{2}\right)^\frac{1}{2}\right) \\
	&\times \|\nabla M(s)\|_{2}\left(\|\nabla M(s)\|_2^2+\|\Delta M(s)\|^2_{2}\right)^\frac{1}{2}\d s+\delta\int_0^t\|\Delta M(s)\|^2_{2} \d s\\
	\leq& c\int_0^t \left(\|\nabla M^1(s)\|^2_{2}+\|\nabla M^2(s)\|^2_{2}+\|\Delta M^1(s)\|^2_2+\|\Delta M^2(s)\|^2_2+\|\nabla M^1(s)\|_2^4 \right. \\
	&\left. +\|\nabla M^2(s)\|_2^4\right.+\|\nabla M^1(s)\|_2^2\|\Delta M^1(s)\|_2^2+\|\nabla M^2(s)\|_2^2\|\Delta M^2(s)\|_2^2\left.\right)\|\nabla M(s)\|^2_2\ds\\
	&+2\delta\int_0^t\|\Delta M(s)\|^2_{2} \d s.
\end{align*}
Next, we have, due to the embedding of $W^{3,2}$ to $W^{2,4}$,
\begin{align*}
|I_{13}|&\leq \int_0^t \|M(s)\|_4\|\Delta M^1(s)\|_4\|\Delta M(s)\|_2 + 2\|M(s)\|_2
\|\Delta M(s)\|_2 \\
& \leq c\int_0^t\left(\|M(s)\|^2_2+\|\nabla M(s)\|^2_2\right)\|\Delta M^1(s)\|^2_4\d s+\delta\int_0^t\|\Delta M(s)\|_2^2\d s\\
&\leq c\int_0^t\left(\|M(s)\|^2_{2}+\|\nabla M(s)\|^2_2\right)\|M^1(s)\|^2_{3,2}+\delta\|\Delta M(s)\|_2^2\d s.
\end{align*}
Finally, we get
\begin{align*}
	|I_{14}|\leq& \int_0^t\int_\Omega (|M^1|+|M^2|)|H_\ext||M||\Delta M|\\
	\leq& 2\int_0^t \|H_\ext(s)\|^2_{4}\|M(s)\|^2_{4}\d s +\delta\int_0^t\|\Delta M(s)\|^2_{2}\d s\\ \leq& c\int_0^t \| H_\ext(s)\|^2_{1,2}\left(\|M(s)\|_2^2+\|\nabla M(s)\|_2^2\right)\d s+\delta\int_0^t\|\Delta M(s)\|^2_{2}\d s
\end{align*}
and 
\begin{align*}
	|I_{15}| \leq \int_0^t c \|M\|_2^2 + \delta \|\Delta M\|_2^2 \d s.
\end{align*}
Therefore, collecting the estimates of the $I_i$'s and choosing $\delta$ suitably small, we deduce from \eqref{FinSum}
\begin{equation*}
	f(t)+g(t)\leq c\int_0^th(s)f(s)\d s\quad t\in (0,T),
\end{equation*} 
where 
the regularity of $(v^1,F^1,M^1)$ and $(v^2,F^2,M^2)$ ensures that $h\in L^1(0,T)$. Since the solutions $(v^1,F^1,M^1)$ and $(v^2,F^2,M^2)$ are raised from the same initial data, we conclude that $v^1=v^2$, $F^1=F^2$, $M^1=M^2$ a.e.\ in $(0,T)\times\Omega$ by the Gronwall lemma. Hence the asserted weak-strong uniqueness follows. \hfill $\Box$

\appendix

\section{Useful inequalities}

In this section we gather two technical propositions. The first one is the classical comparison lemma and in the second one we collect several embedding inequalities.
\begin{lemma}\label{Lem:SolODEComp}
Let $z: [0,t^*) \to \eR_0^+$ solve 
\begin{equation} \label{ODE1}
z'=c(1+z^3), \quad z(0)\geq 0
\end{equation} 
for some $c>0$. Then if $y \geq 0$ solves
\begin{equation} \label{ODE2}
y'\leq c(1+y^3), \quad z(0)\geq y(0) \geq 0,
\end{equation} 
on $[0,t^*)$, it is $y\leq z$ on $[0,t^*)$.
\end{lemma}
\begin{proof}
Taking the difference of \eqref{ODE1} and \eqref{ODE2} yields
$$ z'-y'\geq c (z^3-y^3) = c(z-y)\underbrace{(z^2 + zy +y^2)}_{\geq 0}.$$
Since $z(0)-y(0) \geq 0$ holds by assumption, we conclude $z-y \geq 0$ since $(z-y)'\geq 0$ on $[0,t^*)$.
\end{proof}
\begin{lemma}\label{Lem:Embeddings}
Let $\Omega\subset\eR^2$ be a domain with the smooth boundary. Then   
\begin{align}  
\|f\|_{L^4(\Omega)}&\leq c\|f\|_{L^2(\Omega)}^\frac{1}{2}\|\nabla f\|_{L^2(\Omega)}^\frac{1}{2} \label{Lad2D}
\end{align}
for all $f\in W^{1,2}_0(\Omega)^m$, $m \in \eN$,
\begin{align}
\|f\|_{L^4(\Omega)}&\leq c\left(\|f\|_{L^2(\Omega)}+\|f\|_{L^2(\Omega)}^\frac{1}{2}\|\nabla f\|_{L^2(\Omega)}^\frac{1}{2}\right) \label{GenLad2D}
\end{align}
for all $f\in W^{1,2}(\Omega)^m$,
\begin{align}
\|\nabla f\|_{W^{2,2}(\Omega)}&\leq c\left(\|\nabla f\|^2_{L^2(\Omega)}+\|\nabla\Delta f\|_{L^2(\Omega)}^2\right)^\frac{1}{2},\label{LongIneq}\\
\|\Delta f\|_{L^4(\Omega)}&\leq c\|\Delta f\|_{L^2(\Omega)}^\frac{1}{2}\left(\|\Delta f\|_{L^2(\Omega)}^2+\|\nabla \Delta f\|^2_{L^2(\Omega)} \right)^\frac{1}{4},\label{LadForLaplacian}\\
\|\nabla f\|_{L^\infty(\Omega)}&\leq c\|\nabla f\|_{L^2(\Omega)}^\frac{1}{2}\left(\|\nabla f\|_{L^2(\Omega)}^2
+\|\nabla\Delta f\|_{L^2(\Omega)}^2\right)^\frac{1}{4},\label{AgmonVariant}
\end{align}
for all $f\in W^{3,2}(\Omega)^m$.
Moreover, for all $f\in W^{1,2}_0(\Omega)^m \cap W^{2,2}(\Omega)^m$ we have 
\begin{equation}\label{W22Est}
    \|f\|_{W^{2,2}(\Omega)}\leq c\|\Delta f\|_{L^2(\Omega)},
\end{equation}
and for all $f\in W^{2,2}(\Omega)^m,(\nabla f)n=0$ on $\partial\Omega$ we have
\begin{align}
        \|\nabla f\|_{L^4(\Omega)}&\leq c\|\nabla f\|_{L^2(\Omega)}^\frac{1}{2}\|\Delta f\|_{L^2(\Omega)}^\frac{1}{2},\label{LadMatrix}\\
\|\nabla f\|_{L^6(\Omega)}&\leq c\|\nabla f\|_{L^2(\Omega)}^\frac{1}{3}\left(\|\nabla f\|_{L^2(\Omega)}^2+\|\Delta f\|_{L^2(\Omega)}^2\right)^\frac{1}{3}\label{L6Interp}.
	\end{align}
	\begin{proof}
		Inequalities \eqref{Lad2D}, \eqref{GenLad2D} and \eqref{LadForLaplacian} are variants of Ladyzhenskaya's interpolation inequality. The regularity of a solution to the Laplace equation implies that $\|\cdot\|_{W^{2,2}(\Omega)}$ and $\|\cdot\|_{L^2(\Omega)}+\|\Delta\cdot\|_{L^2(\Omega)}$ are equivalent norms on $W^{2,2}(\Omega)$, which yields \eqref{LongIneq}. Next, \eqref{AgmonVariant} follows from Agmon's inequality in 2D and \eqref{LongIneq}. Inequality \eqref{W22Est} is a consequence of the regularity of a solution to the Laplace equation with the homogeneous Dirichlet boundary condition. Inequality \eqref{LadMatrix} follows from \eqref{GenLad2D} by the application of the Poincar\'e inequality for functions with vanishing normal component of the trace in the form 
		\begin{equation*}
		    \|\nabla f\|_{L^2(\Omega)}\leq c\|\nabla ^2 f\|_{L^2(\Omega)}
		\end{equation*}
		and the fact that $\|\nabla ^2 f\|_{L^2(\Omega)}^2=\|\Delta f\|_{L^2(\Omega)}^2$ due to the homogeneous Neumann boundary condition.
		Finally, \eqref{L6Interp} is a consequence of the Gagliardo-Nirenberg theorem.
	\end{proof}
\end{lemma}

\section*{Acknowledgments} The research leading to the results presented in this paper was supported by DFG grant SCHL 1706/4-1. This paper was finalized while AS was visiting the Isaac Newton Institute for Mathematical Sciences in Cambridge; she is thankful for support and hospitality during the programme DNM supported by EPSRC grant numbers EP/K032208/1 and EP/R014604/1.



\end{document}